\documentclass[a4paper, 12pt]{amsart}
\usepackage[utf8]{inputenc}
\usepackage[english]{babel}
\usepackage[T1]{fontenc}
\usepackage{amsmath, amssymb, amsfonts}
\usepackage{enumerate}
\usepackage{comment}
\usepackage[hmargin=2.25cm, vmargin=2.5cm]{geometry}
\usepackage[colorlinks=true, citecolor=blue]{hyperref}

\title{Positive definite functions and cut-off for discrete groups}
\author{Amaury Freslon}
\address{A. Freslon, Laboratoire de Math\'ematiques d'Orsay, Univ. Paris-Sud, CNRS, Universit\'e Paris-Saclay, 91405 Orsay, France}
\email{amaury.freslon@math.u-psud.fr}
\date{}

\theoremstyle{plain}
\newtheorem{thm}{Theorem}[section]
\newtheorem{prop}[thm]{Proposition}

\newtheorem{lem}[thm]{Lemma}

\theoremstyle{definition}
\newtheorem{de}[thm]{Definition}
\newtheorem*{de*}{Definition}

\theoremstyle{remark}
\newtheorem{rem}[thm]{Remark}

\DeclareMathOperator{\HNN}{HNN}

\DeclareMathOperator{\var}{var}

\newcommand{\C}{\mathbb{C}}

\newcommand{\F}{\mathbb{F}}

\newcommand{\N}{\mathbb{N}}
\newcommand{\R}{\mathbb{R}}
\newcommand{\Z}{\mathbb{Z}}

\newcommand{\dd}{\mathrm{d}}

\begin{document}

\begin{abstract}
We consider the sequence of powers of a positive definite function on a discrete group. Taking inspiration from random walks on compact quantum groups, we give several examples of situations where a cut-off phenomenon occurs for this sequence with respect to the operator norm of the regular representation, including free groups and infinite Coxeter groups. We also give examples of absence of cut-off using free groups again.
\end{abstract}

\maketitle

\section{Introduction}

Positive definite functions have been at the heart of the development of harmonic analysis on locally compact groups since its beginning (see for instance \cite{godement1948fonctions} and \cite{eymard1964algebre}) and are still central in many works on this topic. They are connected to the fundamental notions of amenability, Property (T) and more generally to representation theory and the cohomology of affine isometric actions. In the present paper, we will consider an elementary problem involving these functions. Let $\Gamma$ be a discrete group and let $\varphi : \Gamma\to \C$ be a normalized positive definite function (see Definition \ref{de:positivedefinite}). Under mild hypothesis (see Proposition \ref{prop:convergencecondition}), the sequence $(\varphi^{k})_{k\in \N}$ converges point-wise to a normalized positive definite map and we want to know how fast the convergence is.

If $\Gamma$ is abelian, normalized positive definite functions $\varphi$ corresponds to probability measures $\mu_{\varphi}$ on the Pontryagin dual $\widehat{\Gamma}$ and under this identification, $\varphi^{k}$ corresponds to the $k$-th convolution power of $\mu_{\varphi}$ : $\mu_{\varphi}^{\ast k} = \mu_{\varphi^{k}}$. In other words, our problem is equivalent to the study of a random walk on the dual of $\Gamma$. It is well-known since the founding works of P. Diaconis and his coauthors (see for instance the survey \cite{diaconis1996cutoff}) that such random walks can exhibit a so-called \emph{cut-off phenomenon} : for a number of steps, $\mu_{\varphi}^{\ast k}$ stays at distance almost one from the limiting measure, and then it suddenly converges exponentially fast to it. If now $\Gamma$ is not assumed to be abelian, we can still think of our problem as a random walk on the dual compact \emph{quantum} group $\widehat{\Gamma}$ and ask whether a cut-off phenomenon occurs.

Before explaining the results of this work, let us give a rigorous definition of what we will call a "cut-off phenomenon" in the present article (see Equation \eqref{eq:totalvariation} for the definition of the norm $\|\cdot\|$) :

\begin{de*}
Let $(\Gamma_{N}, \varphi_{N})$ be a family of discrete groups with a state $\varphi_{N}$ on each of them. We say that the sequence $(\varphi_{N}^{k})_{k\in \N}$ has a \emph{pre-cut-off in the window $[k'(N), k(N)]$} if there exists constants $B, B', \lambda$ and $\lambda'$ such that for any $c > 0$ and $N$ large enough,
\begin{equation*}
\|\varphi^{k'(N) - c}_{N} - \delta_{e}\| \geqslant 1 - B'e^{-\lambda' c} \text{ and } \|\varphi^{k(N) + c}_{N} - \delta_{e}\| \leqslant Be^{-\lambda c}.
\end{equation*}
The pre-cut-off is moreover said to be a \emph{cut-off} if $k'(N) = k(N)$.
\end{de*}

The previous definition may seem technical and this is due to the fact that it is stronger than what is often called a cut-off phenomenon. Here is another, more appealing version : for any $\varepsilon > 0$,
\begin{equation*}
\lim_{N\to +\infty}\|\varphi^{(1-\varepsilon)k'(N)}_{N} - \delta_{e}\| = 1 \text{ and } \lim_{N\to +\infty}\|\varphi^{(1+\varepsilon)k(N)}_{N} - \delta_{e}\| = 0
\end{equation*}

One crucial point in the previous definition is that the cut-off phenomenon involves an infinite family of groups. This means that we will have to build families of discrete groups with positive definite functions on them in as natural a way as possible.

Let us now outline the contents of this work. After shortly recalling some preliminary facts in Section \ref{sec:preliminaries}, we give in Section \ref{sec:general} several general results concerning our problem. In particular, we give conditions for convergence to the canonical trace and absolute continuity. We also gather several computations which will be used to produce examples of cut-off. These examples are detailed in Section \ref{sec:cocycle}. They mainly rely on constructions of $1$-cocycles and include free products (Proposition \ref{prop:freeproduct}) and infinite Coxeter groups (Theorem \ref{thm:negativelength}). We end in Section \ref{sec:absence} with an example where there is no cut-off phenomenon, in the sense that the exponential convergence occurs from the first step on for $N$ large enough.

\subsection*{Acknowledgments}

We are grateful to Adam Skalski for the time he spent discussing this work and for his reading of a preliminary version. 

\section{Preliminaries}\label{sec:preliminaries}

Even though our motivation comes from random walks on compact quantum groups, we focus in the present work on the case of duals of discrete groups, for which everything can be written in a classical way through operator algebras. Here is a list of the notions needed (the reader may refer for instance to \cite[Sec F.4]{bekka2008kazhdan} for basics concerning operator algebras associated to discrete groups) : 
\begin{itemize}
\item The Hopf-$*$-algebra associated to $\widehat{\Gamma}$ is the group algebra $\C[\Gamma]$ together with the coproduct given by $\Delta(g) = g\otimes g$ for all $g\in \Gamma$,
\item The Haar state is the canonical trace $\delta_{e}$, the Dirac mass at the neutral element.
\item The GNS construction (see for instance \cite[Thm C.1.4]{bekka2008kazhdan}) with respect to $\delta_{e}$ (which is faithful) yields the left regular representation of $\Gamma$ on $\ell^{2}(\Gamma)$, which in turn provides an embedding of $\C[\Gamma]$ into the group von Neumann algebra $L(\Gamma)$.
\end{itemize}
A probability measure on $\widehat{\Gamma}$ is simply understood as a state on $\C[\Gamma]$ and it is well-known that these admit a group-theoretic description thanks to the notion of positive definite function.

\begin{de}\label{de:positivedefinite}
A function $\varphi : \Gamma\to \C$ is said to be \emph{positive definite} if for any integer $n$, any family $(c_{i})_{1\leqslant i\leqslant n}$ of complex numbers and any $g_{1}, \cdots, g_{n}\in \Gamma$,
\begin{equation*}
\sum_{i, j = 1}^{n}c_{i}\overline{c}_{j}\varphi(g_{i}g_{j}^{-1}) \geqslant 0.
\end{equation*}
If moreover $\varphi(e) = 1$, then $\varphi$ is said to be \emph{normalized}.
\end{de}

Given such a positive definite function $\varphi$, it extends by linearity to a map (again denoted by $\varphi$) on $\C[\Gamma]$ which is positive in the sense that $\varphi(x^{*}x) \geqslant 0$ for any $x\in \C[\Gamma]$. If moreover $\varphi$ is normalized, then this map is a \emph{state}. Moreover, all states arise in that way. Cut-off phenomena only make sense once a notion of convergence for positive definite functions is fixed. A natural choice is the norm $\|.\|_{L(\Gamma)^{*}}$ as linear maps on the von Neumann algebra $L(\Gamma)$, provided the states are indeed defined on $L(\Gamma)$ (this is not automatic, see Proposition \ref{prop:nodensity}). However, for practical reasons (see for instance the proof of Proposition \ref{prop:lowerbound}) we will consider instead the norm
\begin{equation}\label{eq:totalvariation}
\|\cdot\| = \frac{1}{2}\|\cdot\|_{L(\Gamma)^{*}},
\end{equation}
which corresponds to the usual definition of the \emph{total variation distance} in classical probability theory. One may argue that it would be more general to consider the norm $\|\cdot\|_{C^{*}(\Gamma)^{*}}$ as linear forms in the full C*-algebra of $\Gamma$. However, our tools for the computations will require either that the state has an $L^{2}$-density (which forces it to be in $L(\Gamma)^{*}\subset C^{*}(\Gamma)^{*}$ or to be able to do Borel functional calculus, which is not possible in $C^{*}(\Gamma)$. Therefore, it is better to restrict right from the beginning to the von Neumann algebraic setting.

The main tool for proving upper bounds is then the following \emph{Upper Bound Lemma}, which can simply be seen as the contractivity of the inclusion of $L^{1}$ into $L^{2}$ for finite von Neumann algebras. Note that it requires an extra assumption on $\varphi$, namely that it extends to a bounded linear map on $L(\Gamma)$. This is the analogue of the probability measure on $\widehat{\Gamma}$ being absolutely continuous with respect to the Haar measure.

\begin{lem}\label{lem:upperbounddiscrete}
Let $\Gamma$ be a discrete group and let $\varphi$ be a positive definite function on $\Gamma$ which extends to a bounded map on $L(\Gamma)$. Then,
\begin{equation*}
\|\varphi^{k} - \delta_{e}\|^{2} \leqslant \frac{1}{4}\sum_{g\neq e}\vert\varphi(g)\vert^{2k}.
\end{equation*}
\end{lem}

\begin{rem}
As soon as the right-hand side is finite, $\varphi^{k}$ extends to a normal map on $L(\Gamma)$ and consequently the left-hand side is well-defined. Setting it to be infinite otherwise, the inequality makes sense for any positive definite function on $\Gamma$.
\end{rem}

In the sequel we will restrict to discrete groups which are finitely generated and growth considerations will come into the picture so that we recall some elementary facts. If $S$ denotes a symmetric generating set not containing the neutral element $e$, the corresponding word length is defined by
\begin{equation*}
\vert g\vert_{S} = \min\{k\in \N \mid g\in S^{k}\}.
\end{equation*}
Denoting by $B(i)$ the corresponding ball of radius $i$, it follows from submultiplicativity that the sequence $\vert B(i)\vert^{1/i}$ converges to a limite denoted by $\omega(S)$. If this number is $1$, then the group has \emph{subexponential growth} while it is said to have \emph{exponential growth} otherwise.
If we denote by $S(i)$ the sphere of radius $i$ for this metric and by $s_{i}$ its cardinality, it is easy to see that the sequence $s_{i}^{1/i}$ also converges to $\omega(S)$. Moreover,
\begin{equation*}
s_{i}\leqslant \vert S\vert(\vert S\vert-1)^{i-1}.
\end{equation*}
Note that equality for all $i$ holds if and only if the group $\Gamma$ is free on $S$.

\section{General results}\label{sec:general}

In this section we give some general results concerning the total variation distance between powers of a fixed positive definite function and the Haar state $\delta_{e}$. We will in particular prove estimates which will yield cut-off phenomena in the examples of Section \ref{sec:cocycle}.

\subsection{Simple convergence}

Before considering issues related to norm convergence, let us note that it implies simple convergence. We should therefore first understand under which condition the sequence $\varphi^{k}$ converges simply to $\delta_{e}$. For compact groups, it is known that a random walk converges to the Haar state if and only if its support is not contained in a closed subgroup or in a coset with respect to a normal subgroup (see for instance \cite[Thm 3.2.4]{stromberg1960probabilities}). For general compact (and even finite) quantum groups, generalizing this equivalence is still an open problem. In our case however, we can settle it. Let us first give a definition for convenience :

\begin{de}
A normalized positive definite function $\varphi$ on a discrete group $\Gamma$ is said to be \emph{strict} if $\vert\varphi(g)\vert < 1$ for all $g\neq e$.
\end{de}

It is clear that simple convergence to $\delta_{e}$ is equivalent to the initial function being strict. A typical example of a non-strict positive definite function is the counit $\varepsilon : \Gamma\to \C$ sending each $g$ to $1$. More generally, any character (i.e. one-dimensional representation) of $\Gamma$ is not strict. The next result says that this is basically the only obstruction.

\begin{prop}\label{prop:convergencecondition}
Let $\Gamma$ be a discrete group and let $\varphi : \Gamma\to \C$ be a normalized positive definite function. The following are equivalent :
\begin{enumerate}
\item $\varphi$ is not strict,
\item $\varphi$ coincides with a character on a non-trivial subgroup $\Lambda\subset \Gamma$,
\item $\varphi$ is bimodular with respect to a non-trivial subgroup $\Lambda\subset\Gamma$ in the sense that for any $h\in \Lambda$ and $g\in \Gamma$,
\begin{equation*}
\varphi(gh) = \varphi(g)\varphi(h) = \varphi(hg).
\end{equation*}
\end{enumerate}
\end{prop}

\begin{proof}
$(1)\Rightarrow (2)$ : Assume that $\varphi$ is not strict and set
\begin{equation*}
\Lambda = \{h\in \Gamma\mid \vert\varphi(h)\vert = 1\}.
\end{equation*}
The GNS construction provides us with a unitary representation $\pi : \Gamma\to B(H)$ on a Hilbert space $H$ and a unit vector $\xi\in H$ such that for all $g\in \Gamma$, $\varphi(g) = \langle\pi(g)\xi, \xi\rangle$. For $h\in \Lambda$, the Cauchy-Schwarz inequality being an equality, $\pi(h)\xi$ is colinear to $\xi$, hence $\pi(h)\xi = \varphi(h)\xi$. Thus, for any $g\in \Gamma$,
\begin{equation*}
\vert\varphi(gh)\vert = \vert\langle\pi(gh)\xi, \xi\rangle\vert = \vert \varphi(h)\langle\pi(g)\xi, \xi\rangle\vert = \vert\varphi(g)\vert.
\end{equation*}
As a first consequence, $\Lambda$ is stable under multiplication. Since moreover $\varphi(g^{-1}) = \overline{\varphi(g)}$ for any $g\in \Gamma$, we conclude that $\Lambda$ is a subgroup. Moreover, $\varphi : \Lambda\to\C$ is a character and for any $h\in \Lambda$ and $g\in \Gamma$, $\varphi(gh) = \varphi(g)\varphi(h) = \varphi(hg)$.

$(2)\Rightarrow (3)$ : Considering again the GNS representation of $\varphi$, the assertion is equivalent to the fact that $\C\xi$ is  globally invariant under the action of $\Lambda$, from which the bimodularity follows.

$(3)\Rightarrow (1)$ : If $\varphi$ is $\Lambda$-bimodular, then for any $h\in \Lambda$ we have
\begin{equation*}
\vert\varphi(h)\vert^{2} = \overline{\varphi(h)}\varphi(h) = \varphi(h^{-1})\varphi(h) = \varphi(e) = 1
\end{equation*}
so that $\varphi$ is not strict.
\end{proof}


\begin{rem}
If $\Gamma$ is abelian, then there is a probability measure $\mu_{\varphi}$ on the dual compact abelian group $\widehat{\Gamma}$ such that for all $g\in \Gamma$,
\begin{equation*}
\varphi(g) = \int_{\widehat{\Gamma}}\chi(g)\dd\mu_{\varphi}(\chi).
\end{equation*}
The second condition in Proposition \ref{prop:convergencecondition} yields a character $\eta\in \widehat{\Lambda}$ such that $\varphi_{\vert \Lambda} = \eta$, hence for any $h\in \Lambda$,
\begin{equation*}
\int_{\widehat{\Gamma}}(\eta^{-1}\chi)(h)\dd\mu_{\varphi}(\chi) = 1.
\end{equation*}
This implies that the support of $\mu$ is contained in $\eta\Lambda^{\perp}$, where $\Lambda^{\perp} = \{\chi\in \widehat{\Gamma} \mid \Lambda\subset\ker(\chi)\}$ is the annihilator of $\Lambda$. Since any subgroup is normal in the abelian case, we recover the classical criterion.
\end{rem}

\subsection{Absolute continuity}

Once simple convergence is known, we must determine whether $\varphi$ has a normal extension to $L(\Gamma)$. By \cite[Thm V.2.18]{takesaki2002theory}, this is equivalent to the existence of an element $a_{\varphi}$ in the predual $L^{1}(\widehat{\Gamma})$ of $L(\Gamma)$ such that
\begin{equation*}
\varphi(g) = \delta_{e}(a_{\varphi}g)
\end{equation*}
for all $g\in \Gamma$. Computing $L^{1}$-norms is difficult in general, so that we will use $L^{2}$-norms instead. Convergence then involves the rate of decay of $\varphi$ and to make this more precise we introduce the following quantities :
\begin{equation*}
\varphi^{+}(i) = -\inf_{g\in S(i)}\ln(\vert \varphi(g)\vert) \text{ and } \varphi^{-}(i) = -\sup_{g\in S(i)}\ln(\vert \varphi(g)\vert).
\end{equation*}
The definition may seem unnatural but is designed to fit with growth conditions for cocycles, which will be our main source of examples in Section \ref{sec:cocycle}.

\begin{prop}\label{prop:L2density}
Let $\Gamma$ be a discrete group and let $S$ be a finite symmetric generating set. Then,
\begin{itemize}
\item If $\liminf_{i}\frac{\varphi^{-}(i)}{i} > \frac{\ln(\omega(S))}{2k}$, then $\varphi^{k}$ has an $L^{2}$-density, hence also an $L^{1}$-density, with respect to $\delta_{e}$,
\item If $\liminf_{i}\frac{\varphi^{+}(i)}{i} < \frac{\ln(\omega(S))}{2k}$, then $\varphi^{k}$ has no $L^{2}$-density with respect to $\delta_{e}$.
\end{itemize}
\end{prop}

\begin{proof}
We start from the straightforward inequalities
\begin{equation*}
\sum_{i=0}^{+\infty}s_{i} e^{-2k\varphi^{+}(i)}\leqslant \sum_{g\in \Gamma}\vert\varphi(g)\vert^{2k} \leqslant \sum_{i=0}^{+\infty}s_{i} e^{-2k\varphi^{-}(i)}.
\end{equation*}
For the right-hand side, the Cauchy radical test gives a sufficient condition for convergence, namely
\begin{equation*}
\limsup_{i}s_{i}^{1/i}e^{-2k\varphi^{-}(i)/i} < 1.
\end{equation*}
Because $s_{i}^{1/i}$ converges to $\omega(S)$, this gives the first part of the statement. On the other hand, if the series in the middle converges, then so does the one on the left-hand side and using the Cauchy radical test again yields the second part of the statement.
\end{proof}

Proposition \ref{prop:L2density} settles the problem of the existence of $L^{2}$-densities, which is what we need in order to apply Lemma \ref{lem:upperbounddiscrete}. However, it can be that $\varphi$ has an $L^{1}$-density without having an $L^{2}$-density. We will now give a criterion for absolute continuity for some particular class of groups. The idea is to show that $\varphi$ is not bounded on $L(\Gamma)$ by evaluating it on suitable test functions. If a length function on $\Gamma$ is fixed, the natural candidates are the elements
\begin{equation*}
\chi_{i} = \sum_{g\in S(i)} g.
\end{equation*}
Doing this requires a control on the norm of these elements in $L(\Gamma)$ in terms of the sizes of the spheres and such a control is given by the Property of Rapid Decay \cite{jolissaint1990rapidly}. We however nedd an extra positivity assumption on $\varphi$.

\begin{prop}\label{prop:nodensity}
Let $\Gamma$ be a discrete group of exponential growth with the Property of Rapid Decay and let $\varphi$ be a positive definite function on $\Gamma$ taking only positive values. Then, $\varphi^{k}$ extends to a bounded normal functional on $L(\Gamma)$ only if
\begin{equation*}
\liminf_{i}\frac{\varphi_{i}^{+}}{i} \geqslant  \frac{\ln(\omega(S))}{2k}.
\end{equation*}
\end{prop}

\begin{proof}
The Property of Rapid Decay provides us with a polynomial $P$ such that for any element $x\in S(i)$, $\|x\|_{\infty}\leqslant P(i)\|x\|_{2}$. Thus, since $\|\chi_{i}\|_{2} = \sqrt{s_{i}}$,
\begin{equation*}
\frac{\vert\varphi^{k}(\chi_{i})\vert}{\|\chi_{i}\|_{\infty}}\geqslant \frac{1}{P(i)}e^{\frac{\ln(s_{i})}{2}-k\varphi_{i}^{+}}.
\end{equation*}
If the left-hand side is bounded, then there exists $C > 0$ such that for $i$ large enough $\frac{\ln(s_{i})}{2}-k\varphi_{i}^{+}\leqslant C$. Then,
\begin{equation*}
k\frac{\varphi_{i}^{+}}{i}\geqslant \frac{\ln(s_{i})}{2i} - \frac{C}{i}.
\end{equation*}
and because $\Gamma$ is assumed to have exponential growth, the result follows.
\end{proof}

Recall that an amenable group has the Property of Rapid Decay if and only if it has polynomial growth by \cite[Cor 3.1.8]{jolissaint1990rapidly}. The previous proposition therefore only concerns non-amenable groups.

\subsection{Estimates}

In view of the results of the previous section, in order to be able to use Lemma \ref{lem:upperbounddiscrete} we need to make an assumption on the rate of decay of $\varphi$. Proposition \ref{prop:L2density} suggests the condition $\liminf_{i}\varphi^{+}(i)/i > 0$ but we will need a stronger one :

\begin{de}
Let $\Gamma$ be a discrete group and let $\varphi$ be a positive definite function on $\Gamma$. It is said to have \emph{exponential decay} if there exists $\alpha > 0$ and a finite symmetric generating set $S$ such that for all $g\in \Gamma$,
\begin{equation*}
\vert\varphi(g)\vert \leqslant e^{-\alpha\vert g\vert_{S}}.
\end{equation*}
Having exponential decay is independent from the choice of a finite symmetric generating set $S$.
\end{de}

\begin{rem}\label{rem:radicalgrowth}
One may expect a bound of the form $C_{0}e^{-\alpha\vert g\vert_{S}}$ in the above definition, but if we moreover assume that $\varphi$ is strict then this is equivalent to our definition. Indeed, there exists $n_{0}$ and $\alpha' > 0$ such that for $\vert g\vert_{S} > n_{0}$, $C_{0}e^{-\alpha\vert g\vert_{S}}\leqslant e^{-\alpha'\vert g\vert_{S}}$. Moreover, because $\varphi$ is strict,
\begin{equation*}
\alpha'' = \inf_{g\in B(n_{0})\setminus\{e\}}\frac{-\ln(\vert \varphi(g)\vert)}{\vert g\vert_{S}} > 0
\end{equation*}
and setting $\alpha = \min(\alpha', \alpha'')$ yields the result. We are therefore including strictness in the definition of exponential decay.
\end{rem}

By definition, $\alpha$ is less than $\liminf_{i}\varphi^{-}(i)/i$ so that the later is the best potential decay rate, and since it depends on $S$, we will denote it by $\alpha(S)$. With this in hand, we can give a general upper bound statement. The condition for the existence of an $L^{2}$-density suggests that the threshold for exponential convergence should be $\ln(\omega(S))/2\alpha(S)$. However, because $\omega(S)$ is an infimum we cannot use it to bound the series appearing in Lemma \ref{lem:upperbounddiscrete}. We will therefore have to take some room and use $\ln(\vert S\vert - 1)/2\alpha(S)$ instead.

\begin{prop}\label{prop:upperbound}
Let $\Gamma$ be a discrete group with finite symmetric generating set $S$ and let $\varphi : \Gamma\to \C$ be a positive definite function with exponential decay. Then, for any $c > 0$ and $k = \ln(\vert S\vert-1)/2\alpha(S) + c$,
\begin{equation*}
\|\varphi^{k} - \delta_{e}\| \leqslant \frac{e^{-\alpha(S) c}}{\sqrt{2 - 2e^{-\alpha(S) c}}}.
\end{equation*}
\end{prop}

\begin{proof}
Lemma \ref{lem:upperbounddiscrete} yields
\begin{align*}
\|\varphi^{k} - \delta_{e}\|^{2} & \leqslant \frac{1}{4}\sum_{g\neq e}\vert\varphi(g)\vert^{2k} \leqslant \frac{1}{4}\sum_{i=1}^{+\infty}s_{i} e^{-2k \alpha(S)i} \\
& \leqslant \sum_{i=1}^{+\infty}\frac{\vert S\vert}{4}(\vert S\vert - 1)^{i-1} e^{-2k \alpha(S) i} \\
& = \frac{\vert S\vert}{4} e^{-2k\alpha(S)}\frac{1}{1 - (\vert S\vert - 1)e^{-2k\alpha(S)}}.
\end{align*}
For $k = \ln(\vert S\vert)/2\alpha(S) + c$ we therefore get
\begin{equation*}
\|\varphi^{k} - \delta_{e}\|^{2} \leqslant \frac{1}{4}\frac{\vert S\vert}{(\vert S\vert - 1)}\frac{e^{-2\alpha(S) c}}{1 - e^{-2\alpha(S) c}} \leqslant \frac{e^{-2\alpha(S) c}}{2 - 2e^{-2\alpha(S) c}}
\end{equation*}
and the result follows.
\end{proof}

To establish a cut-off phenomenon, it is necessary to have both an upper and a lower bound. Usually, the hard work concerns the upper bound but in our setting we will see that one needs a different argument for each families of group to obtain a lower bound. For the moment we will simply prove a general lower bound which only requires $\varphi$ to take positive values (which will always be the case in our examples). Recall that any finite generating set $S$ of cardinality $n$ gives rise to a quotient map $p : \F_{n}\to \Gamma$. Let $r_{i}$ be the number of reduced words in $\ker(p)$ of length $i$. By \cite[Prop 1]{cohen1982cogrowth}, $r_{i}^{1/i}$ converges (once vanishing terms are removed) to a number $\gamma(S)$ called the \emph{cogrowth rate} of $S$.

\begin{prop}\label{prop:lowerbound}
Let $\Gamma$ be a discrete group with a finite generating set $S$ and let $\varphi$ be a positive definite function on $\Gamma$ such that $\varphi(g)\geqslant 0$ for all $g\in S$. Then, for any $c > 0$ and $k = \ln(\vert S\vert - 1)/2\varphi^{+}(1) - c$,
\begin{equation*}
\|\varphi^{k} - h\| \geqslant 1 - 4\left(2 + 3\frac{\gamma(S)^{2}}{\vert S\vert}\right)e^{-2\varphi^{+}(1)c}.
\end{equation*}
\end{prop}

\begin{proof}
This is where the total variation distance proves useful. In fact, it is easy to prove (see for instance \cite[Lem 2.6]{freslon2017cutoff}) that this norm is equal to the supremum of $\vert\varphi^{k}(p) - \delta_{e}(p)\vert$ over all projections $p\in L(\Gamma)$. Thus, the lower bound will be obtained by evaluating at a suitably chosen spectral projection $p$ of $\chi_{1}$. Choosing the projection first requires some estimates.

According to \cite[Thm 3]{cohen1982cogrowth}, $\|\chi_{1}\|_{\infty} = \gamma(S) + (\vert S\vert - 1)/\gamma(S)$. This only makes sense for groups which are not free on $S$, but the formula can be extended to the latter case by setting $\gamma(S) = \sqrt{\vert S\vert - 1}$ instead of $1$. Since $\chi_{1}$ is self-adjoint, it follows from this that
\begin{equation*}
\var_{\varphi}(\chi_{1})\leqslant \left(\gamma(S) + \frac{\vert S\vert}{\gamma(S)}\right)^{2}.
\end{equation*}
On the other hand, with our assumption we have the estimate
\begin{equation*}
\varphi(\chi_{1}) = \sum_{g\in S} \varphi(g) \geqslant \vert S\vert e^{-k\varphi^{+}(1)}.
\end{equation*}
The lower bound can be obtained from this using the same arguments as in \cite[Prop 3.15]{freslon2017cutoff}. Namely, set $\eta =  \vert S\vert e^{-k\varphi^{+}(1)}/2$ and let us view $\chi_{1}$ as a classical random variable in the algebra $L^{\infty}(\mathrm{Sp}(\chi_{1}))$ which it generates inside $L(\Gamma)$. On the one hand, if $\vert\chi_{1}\vert\leqslant \eta$ then $\vert\chi_{1} - \varphi(\chi_{1})\vert\geqslant \eta/2$ and the probability, with respect to $\varphi$, of this event can be bounded by the Chebyshev inequality. On the other hand, the probability, with respect to $\delta_{e}$, that $\vert\chi_{1}\vert\leqslant \eta/2$ is one minus the probability that $\vert\chi_{1}\vert > \eta/2$ and the latter can also be bounded using the Chebyshev inequality. Putting things together yields
\begin{align*}
\|\varphi^{\ast k} - h\| & \geqslant 1 - \frac{4}{\vert S\vert^{2}}\left(\gamma(S) + \frac{\vert S\vert}{\gamma(S)}\right)^{2}e^{2k\varphi^{+}(1)} - \frac{4}{\vert S\vert}e^{2k\varphi^{+}(1)} \\
& \geqslant 1 - 4\left(\frac{\gamma(S)}{\sqrt{\vert S\vert}} + \frac{\sqrt{\vert S\vert}}{\gamma(S)}\right)^{2}e^{-2\varphi^{+}(1)c} - 4e^{-2\varphi^{+}(1)c} \\
\end{align*}
and the result follows from the fact that $\gamma(S)\geqslant \sqrt{\vert S\vert}$ (see \cite[Thm 1]{cohen1982cogrowth}).
\end{proof}

The problem in the previous statement is that the lower bound involves a term depending on the size of $S$. In order to convert this into a cut-off statement, we need to find families of groups $\Gamma_{N}$ with generating sets $S_{N}$ such that we have a uniform bound on $\gamma(S_{N})^{2}/\vert S_{N}\vert$. In a sense, this means that they are close to free groups. Note however that it is shown in \cite{ollivier2005cogrowth} that for any $\epsilon > 0$, groups with a finite generating set $S$ such that $\gamma(S)\leqslant \sqrt{\vert S\vert} + \epsilon$ are generic in the sense of random groups.

\section{Examples}\label{sec:cocycle}

In this section we will give explicit examples of cut-off phenomena using states coming from $1$-cocycles. An elementary calculation (see for instance \cite[Ex C.2.2.ii]{bekka2008kazhdan}) shows that for any group $\Gamma$, the map $g\mapsto e^{-\|b(g)\|^{2}}$ is positive definite as soon as $b$ is a $1$-cocycle in the following sense :

\begin{de}
Let $\Gamma$ be a discrete group and let $\pi : \Gamma\to B(H)$ be a unitary representation. A \emph{$1$-cocycle associated to $\pi$} is a map $b : \Gamma \to H$ such that for any $g, h\in \Gamma$,
\begin{equation*}
b(gh) = \pi(g)b(h) + b(g).
\end{equation*}
The set of $1$-cocycles associated to $\pi$ is a vector space denoted by $Z^{1}(\Gamma, \pi)$.
\end{de}

Let us denote by $\varphi_{b}$ the state on $\C[\Gamma]$ associated to $b$, i.e.
\begin{equation*}
\varphi_{b} : g\mapsto e^{-\|b(g)\|^{2}}.
\end{equation*}
The exponential decay property translates in this setting into a growth condition for $b$, namel

\begin{de}
Let $\Gamma$ be a discrete group with a finite symmetric generating set $S$. A cocycle $b : \Gamma\to \C$ is said to have \emph{radical growth} if there exists $\alpha > 0$ such that for all $g$ in $\Gamma$,
\begin{equation*}
\|b(g)\|\geqslant \alpha\sqrt{\vert g\vert_{S}}.
\end{equation*}
Having radical growth is a property which does not depend on the choice of a generating set $S$.
\end{de}

\begin{rem}
The existence of a cocycle with radical growth implies that $\Gamma$ has the Haagerup property and that the equivariant Hilbert space compression of $\Gamma$ is at least $1/2$. Moreover, if it is strictly greater than $1/2$, then by \cite[Thm 5.3]{guentner2004exactness} the group $\Gamma$ is amenable. This means that in order to get non-amenable examples, we will need groups with equivariant Hilbert space compression exactly $1/2$.
\end{rem}

\subsection{Sharp cut-off}

We can now give some examples of positive definite functions yielding a sharp cut-off. For this, we need a family $(\Gamma_{N})_{N\in \N}$ of discrete groups together with $1$-cocycles $b_{N}$ and symmetric generating sets $S_{N}$ such that
\begin{itemize}
\item $b_{N}$ has radical growth,
\item $\varphi_{b_{N}}^{+}(1) = \alpha(S_{N})$,
\item $\gamma(S_{N})/\sqrt{\vert S_{N}\vert}$ is uniformly bounded.
\end{itemize}
The simplest instance when the first two conditions are met is when $\|b(g)\|^{2} = \vert g\vert_{S}$, which is possible if and only if (see for instance \cite[Thm C.2.3]{bekka2008kazhdan}) the word length $\vert\cdot\vert_{S}$ is \emph{conditionally negative definite} \cite[Sec 2.10]{bekka2008kazhdan}. But even in that case, there is no reason why the third condition should be met. It turns out however that we can go round it. Note that if quantities are to depend on the size of $S_{N}$, then we will have to assume that the generating set is minimal to get optimal constants.

\begin{thm}\label{thm:negativelength}
Let $(\Gamma_{N}, S_{N})$ be a family of discrete groups with minimal symmetric generating sets $S_{N}$ such that the corresponding length function is conditionally negative definite and let $b_{N}$ be a cocycle such that $\|b_{N}(\cdot)\|^{2} = \vert \cdot\vert_{S_{N}}$. Then, the sequence $(\varphi_{b_{N}}^{k})_{k\in \N}$ has a cut-off at $\ln(\vert S_{N}\vert - 1)/2$ steps.
\end{thm}

\begin{proof}
The upper bound comes directly from Proposition \ref{prop:upperbound} with $\alpha(S) = 1$. As for the lower bound, we will prove that for any $c>0$ and $k = \ln(\vert S_{N} - 1\vert)/2 - c$,
\begin{equation*}
\|\varphi_{b_{N}}^{k} - \delta_{e}\| \geqslant 1 - 8e^{-2c}.
\end{equation*}
As in Proposition \ref{prop:lowerbound}, we will apply the Chebyshev inequality to a spectral projection of $\chi_{1}$. Its expectation satisfies
\begin{equation*}
\varphi_{b_{N}}^{k}(\chi_{1}) = \vert S_{N}\vert e^{-k}.
\end{equation*}
For the variance we will improve the upper bound of Proposition \ref{prop:lowerbound}. Note that because a product of two elements of $S_{N}$ has length two unless the elements are inverse to one another (by minimality of $S_{N}$), we have $\chi_{1}^{2} = \vert S_{N}\vert e + (\vert S_{N}\vert^{2} - \vert S_{N}\vert)\chi_{2}$. Thus,
\begin{align*}
\var_{\varphi_{b_{N}}^{k}}(\chi_{1}) & \leqslant \vert S_{N}\vert + (\vert S_{N}\vert^{2} - \vert S_{N}\vert)e^{-2k} - \vert S_{N}\vert^{2}e^{-2k} \\
& = \vert S_{N}\vert(1-e^{-2k}).
\end{align*}
Since $\delta_{e}(\chi_{1}) = 0$ and $\delta_{e}(\chi_{1}^{2}) = \vert S_{N}\vert$, the standard strategy yields
\begin{align*}
\|\varphi^{k} - \delta_{e}\| & \geqslant 1 - \frac{4\var_{h}(\chi_{1})}{\varphi_{b_{N}}^{k}(\chi_{1})^{2}} - \frac{4\var_{\varphi_{b_{N}}^{k}}(\chi_{1})}{\varphi_{b_{N}}^{k}(\chi_{1})^{2}} \\
& = 1 -  4\frac{e^{2k}}{\vert S_{N}\vert} - 4\frac{e^{2k} - 1}{\vert S_{N}\vert} \\
& \geqslant 1 - 8e^{-2c}.
\end{align*}
\end{proof}

\begin{rem}
The argument above is in fact more general : for any group $\Gamma$ with a minimal symmetric generating set $S$ and a positive definite function $\varphi$, if $\varphi^{-}(2)\geqslant 2\varphi^{+}(1)$ then for any $c > 0$ and $k = \ln(\vert S\vert - 1)/2\varphi^{+}(1) - c$,
\begin{equation*}
\|\varphi_{b_{N}}^{k} - h\| \geqslant 1 - 8e^{-2\varphi^{+}(1)c}.
\end{equation*}
\end{rem}

Here are two families of examples one can build from this :
\begin{itemize}
\item For $N\in \N$, let $\Gamma_{N} = \F_{N}$ be the free group on $N$ generators and take for $S_{N}$ the canonical generators and their inverses. By \cite[Lem 1.2]{haagerup1978example}, the associated word length is conditionally negative definite so that for the corresponding state, Theorem \ref{thm:negativelength} yields a sharp cut-off at $\ln\left(\sqrt{2N-1}\right)$ steps. Note that in that case, the cut-off parameter is indeed equal to the exponential growth rate $\omega(S_{N})$. Moreover, free groups have the Property of Rapid Decay by \cite[Lem 1.4]{haagerup1978example} so that by Proposition \ref{prop:nodensity}, the cut-off happens exactly when the state extends to a bounded normal functional on the von Neumann algebra.
\item Let $S_{N}$ be a set with $N$ elements and let $W_{N}$ be a Coxeter matrix of size $N\times N$. If the corresponding Coxeter group $\Gamma_{N}$ is infinite, it follows from \cite{bozejko1988infinite} that the word length associated to $S_{N}$ is conditionally negative definite. Then, by Theorem \ref{thm:negativelength} there is a sharp cut-off at $\ln(\sqrt{N-1})$ steps. Infinite Coxeter groups also satisfy the Property of Rapid Decay by \cite[Cor 1]{fendler2003simplicity}.
\end{itemize}

\subsection{Free products}

The examples of $\F_{N}$ and $\Z_{2}^{\ast N}$ (this is a Coxeter group) suggest to consider more general free products to build sequences of groups exhibiting a cut-off phenomenon. Indeed, there is a natural way to build a free product of two positive definite functions and the growth of the result is easily controlled. Given positive definite functions $\varphi_{1}$ and $\varphi_{2}$ on groups $\Lambda_{1}$ and $\Lambda_{2}$, any element $g\in \Lambda_{1}\ast\Lambda_{2}$ can be uniquely written as an alternating product $g = h_{1}k_{1}h_{2}k_{2}\cdots h_{n}k_{n}$ with $h_{i}\in \Lambda_{1}$ and $k_{i}\in \Lambda_{2}$ and setting
\begin{equation*}
\varphi(g) = \varphi_{1}(h_{1})\varphi_{2}(k_{1})\cdots \varphi_{1}(h_{n})\varphi_{2}(k_{n})
\end{equation*}
defines a positive definite function (see for instance \cite[Prop 6.2.3]{cherix2001groups}). However, there is to our knowledge no general way of bounding the co-growth of a free product, so that we need another way to bound the norm of $\chi_{1}$ to obtain a lower bound estimate. Let us gather all the sufficient conditions in a single proposition :

\begin{prop}\label{prop:freeproduct}
Let $(\Lambda_{i}, T_{i}, \psi_{i})_{i\in \N}$ be a sequence of discrete groups with a finite symmetric generating set $T_{i}$ and a positive definite function $\psi_{i}$ with exponential decay. Assume moreover that there exist constants $\alpha, \beta, \delta > 0$ such that
\begin{itemize}
\item The decay exponent of $\psi_{i}$ with respect to $T_{i}$ is bounded below by $\alpha$ for all $i$,
\item $\|\chi_{1}^{(i)}\|_{\infty}\leqslant \delta\sqrt{\vert T_{i}\vert}$ for all $i$, where $\chi_{1}^{(i)}$ is the sum of the generators  of $\Lambda_{i}$,
\item $\beta\geqslant \psi_{i}^{+}(1)$ for all $i$.
\end{itemize}
Then, setting $\Gamma_{N} = \Lambda_{1}\ast\cdots \ast\Lambda_{N}$, $S_{N} = T_{1}\sqcup\cdots\sqcup T_{N}$ and $\varphi_{N} = \psi_{1}\ast\cdots \ast\psi_{N}$, the sequence $(\varphi_{N}^{k})_{k\in \N}$ has a pre-cut-off in the window $[\ln(\sqrt{\vert S_{N}\vert})/\beta, [\ln(\sqrt{\vert S_{N}\vert})/\alpha]$.
\end{prop}

\begin{proof}
Consider an element $g = g_{1}\cdots g_{n}\in \Gamma_{N}$ where $g_{j}\in \Lambda_{i_{j}}$ and $i_{j}\neq i_{j+1}$. By construction,
\begin{equation*}
\varphi_{N}(g) = \prod_{j=1}^{n}\psi_{i_{j}}(g_{j})\leqslant \prod_{j=1}^{n}e^{-\alpha_{i_{j}}\vert g_{j}\vert_{S_{i_{j}}}} \leqslant \exp\left(-\alpha\displaystyle\sum_{j=1}^{n}\vert g_{j}\vert_{S_{i_{j}}}\right) = e^{-\alpha\vert g\vert}
\end{equation*}
so that we have a uniform control on the exponential growth rate. Moreover, because elements of length one are exactly generators of the initial groups,
\begin{equation*}
\varphi_{N}^{+}(1) = \sup_{1\leqslant i\leqslant N}\psi_{i}^{+}(1)\leqslant \beta.
\end{equation*}
To be able to conclude we now need a lower bound which can easily be obtained by a free probability argument. More precisely, we can consider the elements $\chi_{1}^{(i)}$ as noncommutative random variables in the noncommutative probability space $(L(\Gamma_{N}), \varphi_{N})$. Because $\varphi_{N}$ is a free product state, the aforementioned variables are freely independent with respect to it. Thus, the variance of their sum is the sum of their variances and
\begin{equation*}
\var_{\varphi_{N}}(\chi_{1}) = \sum_{i=1}^{N}\var_{\psi_{i}}(\chi_{1}^{(i)}) \leqslant \delta^{2}\sum_{i=1}^{N}\vert T_{i}\vert = \delta^{2}\vert S_{N}\vert.
\end{equation*}
The result now follows as in Proposition \ref{prop:lowerbound}.
\end{proof}

The simplest instance where the hypothesis of this proposition are satisfied is when the sequence is constant, i.e.~we consider $\varphi^{\ast N}$ on $\Gamma^{\ast N}$ with generating set $S^{\sqcup N}$. All we need is then a positive definite function with exponential decay, or a cocycle with radical growth. Examples of such cocycles will be given in the next subsection.

\subsection{Geometric cocycles}

We will now give examples of $1$-cocycles with radical growth on discrete groups obtained by geometric means. To this purpose, let us say that a metric space $(X, d)$ has an \emph{equivariant Hilbert embedding with radical growth} if there exists
\begin{itemize}
\item A Hilbert space $H$ together with an affine isometric action of the isometry group of $X$,
\item An equivariant map $f : X\to H$ such that
\begin{equation*}
C^{-1}\sqrt{d(x, y)} - C'\leqslant \|f(x) - f(y)\|\leqslant Cd(x, y) + C'
\end{equation*}
for some constants $C > 0$ and $C'\geqslant 0$.
\end{itemize}
Recall that a group $\Gamma$ is said to act \emph{geometrically} on a metric space $(X, d)$ if it acts properly and cocompactly by isometries. Here is a well-known recipe to produce $1$-cocycles with radical growth :

\begin{prop}\label{prop:geometriccocycle}
Let $\Gamma$ be a finitely generated discrete \emph{torsion-free} group acting geometrically on a geodesic metric space $(X, d)$ having an equivariant Hilbert embedding with radical growth. Then, $\Gamma$ has a $1$-cocycle with radical growth.
\end{prop}

\begin{proof}
By the \v{S}varc-Milnor Lemma (see for instance \cite[Thm 23]{de2000topics}), $\Gamma$ is quasi-isometric to $(X, d)$ so that composing with the equivariant embedding of $X$ we get a $1$-cocycle $b$ on $\Gamma$ satisfying
\begin{equation}\label{eq:almostradicalgrowth}
\|b(g)\|^{2} \geqslant C_{0}\vert g\vert_{S} - C_{1}
\end{equation}
for all $g\in \Gamma$, with constants $C_{0} > 0$ and $C_{1}\geqslant 0$. Set $\varphi_{b}(g) = e^{-\|b(g)\|^{2}}$ and
\begin{equation*}
\Lambda = \{g\in \Gamma \mid b(g) = 0\} = \{g\in \Gamma \mid \vert\varphi_{b}(g)\vert = 1\}.
\end{equation*}
It follows from the proof of Proposition \ref{prop:convergencecondition} that $\Lambda$ is a subgroup. Moreover, because of \eqref{eq:almostradicalgrowth}, $\Lambda$ is finite, hence by torsion-freeness it is the trivial subgroup. Thus, we can conclude by Remark \ref{rem:radicalgrowth} that $\varphi_{b}$ is strict, i.e.~$b$ has radical growth.
\end{proof}

There are many examples of groups and metric spaces satisfying the hypothesis above. This includes trees (see \cite[Sec 2.3]{bekka2008kazhdan}) or even wall spaces (see for instance the comments after \cite[Def 5.1]{shalom2000rigidity}), real or complex hyperbolic spaces $\mathbb{H}^{n}_{\R}$, $\mathbb{H}^{n}_{\C}$ (see \cite[Sec 2.6]{bekka2008kazhdan}). One can also consider product actions on direct products of those spaces. As an illustration, it is shown in \cite[Subsec 6.4]{dreesen2011equivariant} that Baumslag-Solitar groups $BS(p, q)$ with $p, q > 1$ act geometrically on the product of the corresponding Bass-Serre tree and the real hyperbolic plane. Moreover, we have the following "permanence properties" for the existence of a $1$-cocycle with radical growth :
\begin{itemize}
\item If $\Gamma_{1}$ and $\Gamma_{2}$ have a $1$-cocycle with radical growth and $\Lambda$ is a common finite subgroup, then the proof of \cite[Thm 4.7]{dreesen2011hilbert} shows that the amalgamated free product $\Gamma_{1}\ast_{\Lambda}\Gamma_{2}$ has a $1$-cocycle with radical growth,
\item If $\Gamma$ is a finitely generated group together with a subgroup $\Lambda$ and a monomorphism $\theta : \Lambda\to\Gamma$ such that $\Lambda\cup\theta(\Lambda)$ generates a finite subgroup of $\Gamma$, then any $1$-cocycle with radical growth on $\Gamma$ yields a $1$-cocycle with radical growth on $\HNN(\Gamma, \Lambda, \theta)$ by the proof of \cite[Thm 4.9]{dreesen2011hilbert}.
\end{itemize}
This gives a wealth of examples, including for instance any group built from Baumslag-Solitar groups, free groups, surface groups and infinite Coxeter groups that one can iterate.

\section{Absence of cut-off}\label{sec:absence}

In this final section we will give a family of examples where there is no cut-off phenomenon, in the sense that for $N$ large enough, exponential convergence occurs from the first step on. This involves states not coming from cocycles as before. The general form of a state $\varphi$ on a C*-algebra $A$ is given by the GNS construction (see \cite[Thm C.1.4]{bekka2008kazhdan}), which yields a Hilbert space $H$, a unitary representation $\pi : A\to B(H)$ and a unit vector $\xi\in H$ such that $\varphi(x) = \langle \pi(x)\xi, \xi\rangle$. If $A = C^{*}(\Gamma)$, representations are in one-to-one correspondence with representations of $\Gamma$. From now on we will denote by $\varphi_{\pi, \xi}$ the state $\varphi_{\pi, \xi} : g\mapsto \langle \pi(g)\xi, \xi\rangle$.

It is difficult to give a general criterion to know whether $\varphi_{\pi, \xi}$ will have exponential decay since this heavily depends on the representation $\pi$ and therefore on the structure of the group $\Gamma$. The simplest case is certainly when $\pi$ is the regular representation on $\ell^{2}(\Gamma)$. In that case, we write $\varphi_{\xi}$ for $\varphi_{\textrm{reg}, \xi}$. As a case study, we will focus on free groups and restrict our attention to a particular class of vectors :

\begin{de}
A vector $\xi\in \ell^{2}(\Gamma)$ is said to be \emph{radial} if it is of the form
\begin{equation*}
\xi = \sum_{i=0}^{+\infty}\lambda_{i}\chi_{i}.
\end{equation*}
\end{de}

We will prove that such a state never exhibits a cut-off phenomenon. For convenience, let us set $\eta_{i} = (\vert S\vert - 1)^{i/2}\lambda_{i}$, so that the assumption that $\xi\in \ell^{2}(\Gamma)$ is equivalent to $(\eta_{i})_{i\in \N}\in \ell^{2}(\N)$. Note moreover that $\|(\eta_{i})_{i\in \N}\|_{2}\leqslant \|\xi\|_{2} = 1$.

\begin{thm}\label{thm:radialvectors}
Let $\Gamma_{N}$ be the free group on $N$ generators with its canonical generating set $S_{N}$. If $\xi$ is a radial vector, then the state $\varphi_{\xi}$ has exponential decay for $N\geqslant 3$ but the sequence $(\varphi_{\xi}^{k})_{k\in \N}$ has no cut-off.
\end{thm}

\begin{proof}
We have
\begin{equation*}
\varphi_{\xi}(g) = \langle \pi(g)\xi, \xi\rangle = \sum_{i, j=0}^{+\infty} \lambda_{i}\overline{\lambda}_{j}\langle g.\chi_{i}, \chi_{j}\rangle = \sum_{i, j=0}^{+\infty} \lambda_{i}\overline{\lambda}_{j}\vert\left(g.S(i)\right)\cap S(j)\vert.
\end{equation*}
The set appearing above is empty unless $j \leqslant i + \vert g\vert$. Because we are considering free groups, an element $h\in S(i)$ such that $gh\in S(j)$ must be of the form $g_{p}^{-1}\cdots g_{p-t+1}^{-1}w$ where $t = (\vert g\vert + i - j)/2$, $g = g_{1}\cdots g_{p}$ and $w\in S(i-t)$ does not start with $g_{p-t+1}^{-1}$. Thus, there are exactly $(\vert S_{N}\vert - 1)^{i-t}$ such elements and
\begin{align*}
\varphi_{\xi}(g) & = \sum_{i=0}^{+\infty}\sum_{t=0}^{\min(i, \vert g\vert)}\lambda_{i}\overline{\lambda}_{i+\vert g\vert - 2t}(\vert S_{N}\vert - 1)^{i-t} \\
& = (\vert S_{N}\vert - 1)^{-\vert g\vert/2} \sum_{i=0}^{+\infty}\eta_{i}\left(\sum_{t=0}^{\min(i, \vert g\vert)}\overline{\eta}_{i+\vert g\vert-2t}\right).
\end{align*}
For $t\in \Z$, let $T_{t}$ be the shift operator on $\ell^{2}(\N)$ sending $\delta_{n}$ to $\delta_{n-t}$ if $t\leqslant n$ and $0$ otherwise. Then,
\begin{equation*}
\left(\sum_{t=0}^{\min(i, \vert g\vert)}\eta_{i+\vert g\vert-2t}\right)_{i\in \N} = \sum_{t=0}^{\vert g\vert}T_{\vert g\vert - 2t}\left((\eta_{i})_{i\in \N}\right)
\end{equation*}
and since the shifts have norm one, this vector has $\ell^{2}$-norm at most $(\vert g\vert+1)\|(\eta_{i})_{i\in \N}\|_{2}\leqslant \vert g\vert + 1$ so that by the Cauchy-Scharz inequality,
\begin{equation*}
\vert\varphi_{\xi}(g)\vert\leqslant (\vert g\vert + 1)(\vert S\vert - 1)^{-\vert g\vert/2} \leqslant e^{(1-\ln(\vert S_{N}\vert - 1))\vert g\vert/2}.
\end{equation*}
It follows that for $N\geqslant 3$, the state $\varphi_{\xi}$ has exponential decay with rate $(\ln(\vert S_{N}\vert - 1) - 1)/2$. But then, 
\begin{equation*}
\frac{\ln(\vert S_{N}\vert - 1)}{2\alpha} = \frac{\ln(2N-1)}{\ln(2N-1) - 1} \to 1
\end{equation*}
so that for $N$ large enough the convergence is exponential from the first step on and there is no cut-off phenomenon.
\end{proof}

The functions in the statement are precisely those which are radial and associated with the regular representation. We cannot exclude that for some representation $\pi$ which is not contained in the regular one, and a vector $\eta$, exists $k$ such that $\varphi_{\pi, \eta}^{k} = \varphi_{\xi}$, in which case $\varphi_{\pi, \eta}$ would exhibit a cut-off phenomenon. In that case our result shows that exponential converge occurs right when the function becomes associated with the regular representation.

\bibliographystyle{amsplain}
\bibliography{../../../quantum}

\end{document}